\documentclass[11pt,preprint,a4paper]{elsarticle}
\usepackage{amsmath}
\allowdisplaybreaks[4]
\usepackage{float}
\usepackage{caption}
\usepackage{color}
\usepackage{subfigure}
\usepackage{graphicx}
\usepackage{amssymb}
\usepackage{amsthm}
\usepackage{setspace}
\usepackage{mathrsfs}
\usepackage{amsfonts}
\usepackage{galois}
\usepackage{epic}
\usepackage{CJK}
\usepackage[margin=25mm]{geometry}
\newtheorem{thm}{Theorem}[section]

\newtheorem{lem}[thm]{Lemma}
\newtheorem{rem}[thm]{Remark}

\newtheorem{definition}[thm]{Definition}
\newtheorem{pro}[thm]{Proposition}
\numberwithin{equation}{section}
\begin{document}
\begin{frontmatter}
\title{Multiplicity of closed characteristics on $P$-symmetric compact convex hypersurfaces in $\mathbb{R}^{2n}$}
\author{Lei Liu}\ead{201511242@mail.sdu.edu.cn}\author{Li Wu}\ead{201790000005@sdu.edu.cn}
\address{School of Mathematics Department, Shandong University, Jinan 250001, P.R. China}
\begin{abstract}
There is a long standing conjecture that there are at least $n$ closed characteristics for any compact convex hypersurface $\Sigma$ in $\mathbb{R}^{2n}$, and the symmetric case, i.e. $\Sigma=-\Sigma$, has already been proved by C. Liu, Y. Long and C. Zhu in [Math. Ann., 323(2002), pp. 201-215]. In this paper, we extend the result in that paper to the $P$-symmetric case $\Sigma=P\Sigma$ for a certain class of symplectic matrix $P$, and prove that there are at least $[\frac{3n}{4}]$ closed characteristics on $\Sigma$ for any positive integer $n$, where $[a]:=\sup\{l\in\mathbb{Z},l\leq a\}$. To obtain our result, the key problem is to estimate (\ref{equ the estimation}) in which the method is based on the theorem called Common Index Jump Theorem. By using the Bott-type iteration formulas of Maslov index and Maslov-type index for a certain kind of iteration symplectic path, we provide the some new estimations (\ref{i+2S-v inequality 3}-\ref{i+2S-v inequality 4}), which are not considered in other papers.
\end{abstract}

\end{frontmatter}
\renewcommand\thefootnote{}
\footnote{The author is partially supported by NSFC(No.11425105).}
\section{Introduction}
This paper deal with the multiplicity of closed characteristics on $P$-symmetric compact convex hypersurfaces in $\mathbb{R}^{2n}$. For each $C^2$-compact convex hypersurface $\Sigma\in\mathbb{R}^{2n}$ surrounding $0$, we will consider the following problem:
\begin{equation}\label{initial problem}
\left\{
\begin{aligned}{}
&\dot{y}(t)=JN_\Sigma(y(t)),\ y(t)\in\Sigma,\ \forall t\in\mathbb{R},\\
&y(\tau)=y(0), \ \tau>0,
\end{aligned}\right.
\end{equation}
with the standard symplectic matrix $J$, i.e. $\begin{pmatrix}0&-I_n\\I_n&0\end{pmatrix}$, and the outward normal unit vector $N_\Sigma(x)$ of $x\in \Sigma$. A solution $(\tau,y)$ of (\ref{initial problem}) is called a closed characteristic on $\Sigma$, and we call it a prime closed characteristic if $\tau$ is the minimal period. If two closed characteristics $(\tau,y)$ and $(\sigma,z)$ are not completely overlapping, then they are called geometrically distinct. For the set of closed characteristics and geometrically distinct ones $$[(\tau,y)]=\{(\sigma,z)\in\mathcal{J}(\Sigma)|y(\mathbb{R})=z(\mathbb{R})\},$$
we denote them by $\mathcal{J}(\Sigma)$ and $\hat{\mathcal{J}}(\Sigma)$, respectively.

In the last century, this problem attracted the attentions of many mathematicians. The milestone of this problem was made by P.H. Rabinowitz and A. Weinstein in 1978 \cite{Rabinowitz.1978,Weinstein.1978}, who proved that $${}^\#\hat{\mathcal{J}}(\Sigma)\geq 1,\ \forall\Sigma\in\mathcal{H}(2n).$$
The notation $\mathcal{H}(2n)$ denotes the collection of all hypersurfaces as considered in (\ref{initial problem}). After this, I. Ekeland, H. Hofer, L. Lassoued and A. Szulkin \cite{Ekeland.Hofer.1987,Ekeland.Lassoued.1987,A.Szulkin.1988} provide a stronger result that
\begin{eqnarray*}
{}^\#\hat{\mathcal{J}}(\Sigma)\geq 2,\ \forall \Sigma\in\mathcal{H}(2n),\ n\geq 2.
\end{eqnarray*}
This result was improved greatly by Y. Long and C. Zhu \cite{Long.zhu.2002}. They showed that
$${}^\#\hat{\mathcal{J}}(\Sigma)\geq [n/2]+1,\forall \Sigma\in \mathcal{H}(2n).$$
Later, the case $n=3,4$ for this conjecture were proved in \cite{W.H.L.2007} and \cite{Wang.2016}, respectively.

Besides, there are also plenty of results for special compact convex hypersurfaces. In \cite{Liu.Long.Zhu.2002}, C. Liu, Y. Long and C. Zhu gave the first surprising result that
\begin{eqnarray}
{}^\#\hat{\mathcal{J}}(\Sigma)\geq n,\ \forall \Sigma=-\Sigma. \label{equ Long and Zhu 2002}
\end{eqnarray}
This is the only result that proves the conjecture in high dimension. After that, Y. Dong and Y. Long \cite{Dong.Long.2004} studied the $\mathcal{P}$-symmetric case for $$\mathcal{P}=\mathrm{diag}(-I_{n-\kappa},I_{\kappa},-I_{n-\kappa},I_{\kappa}),\ \kappa\in\{1,\cdots,n\}.$$ Under certain assumptions about $\Sigma$, it holds that
\begin{eqnarray*}
{}^\#\hat{\mathcal{J}}(\Sigma)\geq n-2\kappa,\ \Sigma=\mathcal{P}\Sigma.
\end{eqnarray*}
The $P$-symmetric case was studied by D. Zhang in \cite{Zhang.2013}. He considered the symplectic and orthogonal matrix $P$ with satisfying $P^r=I_{2n}$ for some integer $r>1$, and proved there are at least two geometrically distinct closed characteristics $(\tau_j,x_j)$ satisfying that
$$x_j(t+\frac{\tau_j}{r})=Px_j(t),\ t\in \mathbb{R},\ j=1,2.$$
For other kinds of special hypersurfaces such as brake symmetric, pinched or star-shaped hypersurfaces, we refer to \cite{Bolotin.1978,Duan.Liu.2017,Liu.2014,Liu.Zhang.2014,Liu.Zhu.2018,Liu.Long.Wang.2014,Long.Zhang.Zhu.2006,Rabinowitz.1987,A.Szulkin.1989,Wang.2017} and it's references.

However there are still no other results about the total number of closed characteristics on $P$-symmetric ones yet. Therefore, in this paper, we will focus on the $P$-symmetric case. We define the symplectic group as
$$\mathrm{Sp}(2n)=\{M\in GL(2n,\mathbb{R})|M^TJM=J\}.$$
For $\Sigma\in\mathcal{H}_P(2n):=\{\Sigma\in\mathcal{H}(2n)|x,Px\in\Sigma,\forall x\in\Sigma\},\ P\in\mathrm{Sp}(2n)$, we call $\Sigma$ $P$-symmetric, and we also call a closed characteristic $(\tau,x)$ $P$-symmetric if $(\tau,x)$ belongs to the set $$\mathcal{J}_P(\Sigma):=\{(\tau,x)\in\mathcal{J}(\Sigma)|x(\mathbb{R})=Px(\mathbb{R})\}.$$

Then we prove the following main results.

\begin{thm}\label{main theorem}
Assume that $P\in\mathrm{Sp}(2n)$ and $\Sigma\in\mathcal{H}_P(2n)$. If $P$ is similar to the matrix
\begin{eqnarray}
R(-\theta)^{\diamond n-[\frac{n}{2}]}\diamond R(\theta)^{\diamond[\frac{n}{2}]}\label{assumption 1}
\end{eqnarray}
with $\diamond$ defined as (\ref{diamond product}), then
\begin{eqnarray}\label{equ2 thm1.1}
{}^\#\hat{\mathcal{J}}(\Sigma)\geq [\frac{3n}{4}],
\end{eqnarray}
where $R(\theta)=\begin{pmatrix}\cos\theta& -\sin\theta\\ \sin\theta& \cos\theta\end{pmatrix}$, $e^{im\theta}=1$ for some integer $m>1$ and $\frac{\theta}{2\pi}\notin\mathbb{Z}$.
\end{thm}
Actually, this estimation is followed directly by Theorem \ref{main theorem 1}, and we can also obtained ${}^\#\hat{\mathcal{J}}(\Sigma)\geq n$ when $m$ is even, i.e. the result in \cite{Liu.Long.Zhu.2002}. Besides, according to Proposition \ref{properties of symmetric orbit}(3) below, we find out for any $P$-symmetric closed characteristic $(\tau,x)$, there exist $\l(x)\in\{1,\cdots,m-1\}$ such that $x(t)=Px(t+\frac{l(x)\tau}{m})$. However $l$ is indeterminate for each $(\tau,x)$. If we fix $l(x)\equiv 1$ for all prime closed characteristics, the following result will hold.
\begin{thm}\label{main theorem 2}
Let $P\in \mathrm{Sp}(2n)$ and integer $m>1$. Assume that the following conditions hold:\\
(i) $P$ is similar to the matrix $R(-\theta)^{\diamond n}$ with $\theta\in(0,\pi]$ and $e^{im\theta}=1$,\\
(ii) $x(t)=Px(t+\frac{\tau}{m})$ for any prime closed characteristic $(\tau,x)\in\mathcal{J}_P(\Sigma)$ and $\ t\in\mathbb{R}$, \\
then $${}^\#\hat{\mathcal{J}}(\Sigma)\geq n.$$
\end{thm}
\begin{rem}\label{remark 1.2}
If $P$ is orthogonal, we can also replace Theorem 1.1 in \cite{Liu.Tang.2015}, i.e. the Bott-type fomula for Maslov index, with Theorem 1.1 in \cite{Hu.Sun.2009} to prove these results.
\end{rem}

An outline of this paper is as follows. In Section 2, the Maslov-type index and Maslov index are briefly introduced, and then we list some properties of splitting numbers and Krein type numbers. In Section 3, firstly, we provide the properties of symmetric closed characteristics, which have been considered by X. Hu and S. Sun in \cite{Hu.Sun.2009} for orthogonal symplectic matrix $P$. Then we transfer the multiplicity problem into the estimation (\ref{equ the estimation}) by using the Common Index Jump Theorem, i.e. Lemma \ref{common index jump theorem}. In Section 4, by using the Bott-type iteration formulas, especially, Theorem 1.1 in \cite{Liu.Tang.2015}, we provide some new estimations (\ref{i+2S-v inequality 3}-\ref{i+2S-v inequality 4}) and prove the main results.
\section{Known properties}
In this section, we first introduce the Maslov-type index and Maslov index briefly. Then we will list some useful properties of Splitting numbers and Krein type numbers which will play important roles in this paper.

Let $\mathcal{P}_\tau(2n)$ be the collection of symplectic paths in $C([0,\tau],\mathrm{Sp}(2n))$ starting from $I_{2n}$ and
\begin{align*}
\mathrm{Sp}(2n)^0_\omega=&\{M\in\mathrm{Sp}(2n)|\det(M-\omega I_{2n})=0\},\ \mathrm{Sp}(2n)^*_\omega=\mathrm{Sp}(2n)\setminus\mathrm{Sp}(2n)^0_\omega
\end{align*}
for any $n\in\mathbb{N},\ \omega\in\mathbf{U}:=\{z\in\mathbb{C},|z|=1\}$ and $\tau>0$. The vector space $V$ is called symplectic if $V$ endows with a nondegenerate and closed two form $\omega_V$. Let $\tilde{J}:V\rightarrow V$ be a complex structure, i.e. $J^2=-id$. $J$ is said to be compatible with $\omega_V$ if
$$\omega_V(Ju,Jv)=\omega_V(u,v),\ \omega_V(v,Jv)>0,\forall u,v\in V.$$
A subspace $\Lambda\subset V$ is called Lagrangian if $\omega_V|_\Lambda=0$ and $\dim\Lambda=n$. Let $(\mathbb{R}^{2n},\omega_{st})$ be the standard symplectic space, where the standard symplectic form $$\omega_{st}((x_1,y_1),(x_2,y_2)):=(x_1,y_2)-(x_2,y_1),\ \forall (x_i,y_i)\in\mathbb{R}^{2n},i=1,2.$$ Each Lagrangian $\Lambda$ in $(\mathbb{R}^{2n},\omega_{st})$ will possess a Lagrangian frame $Z(\Lambda)=(X,Y)^T:\mathbb{R}^n\rightarrow\mathbb{R}^{2n}$ whose image is $\Lambda$, where $X,Y$ are $n\times n$ matrices satisfying $Y^TX=X^TY$. Let Lag$(n)$ be the Lagrangian Grassmannian which is the collection of all Lagrangian in $(\mathbb{R}^{2n},\omega_{st})$.

We introduce the maslov-type index and Maslov index as follows.
\begin{definition}\label{definition of index}
Let $\gamma\in\mathcal{P}_\tau(2n)$ and $\omega\in\mathbf{U}$, the Maslov-type index ($\omega$ index) is a integer-pair value map
$$(i_\omega,\nu_\omega):\mathcal{P}_\tau(2n)\rightarrow \mathbb{Z}\times\{0,1,\cdots,2n\},$$
in which
$$\nu_\omega(\gamma):=\dim_{\mathbb{C}}\ker_{\mathbb{C}}(\gamma(\tau)-\omega I_{2n})$$
and $i_\omega$ is uniquely characterized by five properties of Theorem 6.2.7 in \cite{Y.Long.2002}.

Let $\mathcal{P}([a,b],V)$ be the collection of pairs of Lagrangian paths in V. For any pair of Lagrangian paths $(\Lambda_1,\Lambda_2)$, the Maslov index (CLM index) is an integer value map
$$\mu:\mathcal{P}([a,b],V)\rightarrow \mathbb{Z}$$
uniquely characterized by properties I-VI in \cite{C.L.M.1994}.
\end{definition}
\begin{rem}
From the Theorem 2.6.6 in \cite{Y.Long.2002}, we known that for each symplectic matrix $M$, the symplectic path $Me^{Jt}\in\mathrm{Sp}(2n),\ t\in\mathbb{R},$ will transverse the singular cycle $\mathrm{Sp}(2n)^0_{\omega}$ for any $\omega\in\mathbf{U}$. Here we call the directions of path $Me^{Jt}$ and $Me^{-Jt}$ are the positive and negative direction of $M$, respectively. For any symplectic path $\gamma\in\mathcal{P}_\tau(2n)$, the Maslov-type index $i_\omega(\gamma)\in\mathbb{Z}$ in \cite{Y.Long.2002} is actually defined by counting how many half-turns $\gamma$ winds after perturbing the end point $\gamma(\tau)$ along the negative direction a little.

For any symplectic path $\gamma$ which may not start from $I_{2n}$, the Maslov index is originally defined in Theorem 1.1 in \cite{C.L.M.1994}. However , if $\Lambda_1\equiv V$ is contant, we can also explain $\mu(\Lambda_1,\Lambda_2)$ as $[e^{-\epsilon \tilde{J}}\Lambda:\Sigma(V)]$, which is the number of points (counting the multiplicity) that the path $e^{-\epsilon \tilde{J}}\Lambda$, i.e. perturbing the whole curve along the negative direction a little, intersects on the Maslov cycle $$\Sigma(V):=\{\Lambda\in\mathrm{Lag}(2n),\Lambda\cap V\neq\{0\}\}\subset\mathrm{Sp}(2n).$$ Where $\epsilon>0$ is small enough and the matrix $\tilde{J}$ is compatible with the symplectic form $\omega_V$. All the details of this two definitions can be found in \cite{Long.zhu.2000}, \cite{Y.Long.2002}, \cite{C.L.M.1994} and \cite{Hu.Sun.2009}. Especially, when $(V,\omega_V)=(\mathbb{R}^{4n},\omega_{st}\oplus(-\omega_{st}))$, the graph of any symplectic path $\gamma\in C([a,b],\mathrm{Sp}(2n))$ will become a Lagrangian path in $(\mathbb{R}^{4n},\omega_{st}\oplus(-\omega_{st}))$, i.e.
$$Gr(\gamma(t)):=\{(x,\gamma(t)x),x\in \mathbb{R}^{2n}\}\in\mathrm{Lag}(2n), \forall t\in[a,b].$$
Then we can define the Maslov index of $\gamma$ respect to $P\in\mathrm{Sp}(2n)$ as
$$\mu(Gr(P^T),Gr(\gamma))=[e^{-\epsilon \tilde{J}}Gr(\gamma):\Sigma(Gr(P^T))]=[Gr(e^{-2\epsilon J}\gamma):\Sigma(Gr(P^T))],$$ in which $\tilde{J}$ will be $J\oplus(-J)$ and is compatible with the symplectic form $\omega_{st}\oplus(-\omega_{st})$.
\end{rem}
By these definitions above, the relations between this two indies given by Lemma 4.6 in \cite{Hu.Sun.2009} still hold when $P$ is not just orthogonal but symplectic.
\begin{lem}\label{lem relation between two indies}
Let $P\in \mathrm{Sp}(2n)$, $\gamma\in\mathcal{P}_{\tau}(2n)$, we have
\begin{eqnarray}
\mu(Gr(\omega I),Gr(\gamma(t)))=
\left\{
\begin{aligned}{}
&i_1(\gamma)+n,\ \omega=1 \\
&i_\omega(\gamma),\ \omega\in\mathbf{U}\setminus\{1\},
\end{aligned}\right.\label{relation between two indies}
\end{eqnarray}
and
\begin{eqnarray}
\mu(Gr(\omega P^{-1}),Gr(\gamma(t)))=\mu(Gr(\omega),Gr(P\gamma(t)))=i_\omega(\tilde{\gamma}\ast\xi)-i_\omega(\xi)\label{relation between two indies 2}
\end{eqnarray}
for any path $\xi\in\mathcal{P}_\tau(2n)$ connecting $I_{2n}$ to $P$, where $\tilde{\gamma}=P\gamma$ and
\begin{eqnarray}\label{path connecting}
\tilde{\gamma}\ast\xi:=\left\{
\begin{aligned}{}
&\xi(2t),\ t\in[0,\frac{\tau}{2}]\\
&\tilde{\gamma}(2t-\tau),\ t\in[\frac{\tau}{2},\tau].
\end{aligned}\right.\label{def of joint paths}
\end{eqnarray}
\end{lem}
Note that the identity (\ref{relation between two indies 2}) is exactly same as the Maslov index $i^{P^{-1}}_\omega(\gamma)$ defined in \cite{Liu.Tang.2015}.

Denote $M$ as the end point $\gamma(\tau)$ of $\gamma$, the Krein type numbers $(P_\omega(M),Q_\omega(M))$ are defined by the total multiplicities of positive and negative eigenvalues of $\sqrt{-1}J$ restricted on $E_\omega=\ker(M-\omega I_{2n})^{2n}$, respectively. Since $\sqrt{-1}J$ is nondegenerate, then we have
\begin{eqnarray}
P_\omega(M)+Q_\omega(M)=\dim(E_\omega)\geq \nu_\omega(M),\ \forall \omega\in\mathbf{U}.\label{identity of Klein number}
\end{eqnarray}
The $m$-iteration path of $\gamma$ is defined by
\begin{eqnarray*}
\gamma^m(t)=\gamma(t-j\tau)\gamma(\tau)^j,&\forall j\tau\leq t\leq (j+1)\tau,&j=0,\cdots,m-1.
\end{eqnarray*}
In addition, the splitting numbers of $\gamma$ were introduced in Section 9.1 of \cite{Y.Long.2002} as
\begin{eqnarray}
S^{\pm}_{M}(\omega)=\displaystyle\lim_{\epsilon\rightarrow\pm 0}(i_{\exp(\sqrt{-1}\epsilon)\omega}(\gamma)-i_\omega(\gamma)),\label{equ initial def of splitting number}
\end{eqnarray}
which only depend on the end point $M$, actually. For splitting numbers, we have
\begin{pro}\label{pro of initial splitting number}
Let $M\in\mathrm{Sp}(2n)$, $\omega\in\mathbf{U}$ and $0\leq\bar{\theta}\leq\hat{\theta}\leq2\pi$. $\sigma(M)$ donates the spectral set of $M$. Then there hold
\begin{eqnarray}
&S^\pm_M(\omega)\geq 0,\forall \omega\in\mathbf{U},\ S^\pm_M(\omega)=0, \omega\notin\sigma(M), \label{pro splitting number 1}\\
&S^\pm_M(\omega)=S^\mp_M(\overline{\omega}),\ \nu_\omega(M)=\nu_{\overline{\omega}}(M), \label{pro splitting number 2}\\
&P_\omega(M)-S^+_M(\omega)=Q_\omega(M)-S^-_M(\omega)\geq 0, \label{pro Splitting number 8}\\
&S^+_M(\omega)+S^-_M(\omega)\leq\dim\ker(M-\omega I_{2n})^{2n},\ \omega\in\sigma(M), \label{pro splitting number 3}\\
&(S^+_{I_2}(1),S^-_{I_2}(1))=(1,1),\ (S^+_{-I_2}(-1),S^-_{-I_2}(-1))=(1,1), \label{pro splitting number 5}\\
&(S^+_{R(\theta)}(e^{\sqrt{-1}\theta}),S^-_{R(\theta)}(e^{\sqrt{-1}\theta}))=(0,1)\ and \label{pro splitting number 6}\\
&(S^+_{R(-\theta)}(e^{\sqrt{-1}\theta}),S^-_{R(-\theta)}(e^{\sqrt{-1}\theta}))=(1,0), \forall \theta\neq\pm 1, \label{pro splitting number 7}\\
&S^\pm_{M_1\diamond M_2}(\omega)=S^\pm_{M_1}(\omega)+S^\pm_{M_2}(\omega),\ M_j\in\mathrm{Sp}(2n_j),j=1,2\ and\ n=n_1+n_2, \label{pro splitting number 11}\\
&S^+_{M}(\omega)=S^+_{N}(\omega), \ for\ any\ N\in\mathrm{Sp}(2n)\ similar\ to\ M,\label{pro similarity}\\
&0\leq \nu_\omega(M)-S^-_M(\omega)\leq P_\omega(M),\ 0\leq \nu_\omega(M)-S^+_M(\omega)\leq Q_\omega(M), \label{pro 0<v-Spm(omega)<P,Q}\\
&(P_\omega(M),Q_\omega(M))=(Q_{\overline{\omega}}(M),P_{\overline{\omega}}(M)), \label{pro splitting number 12}\\
&\displaystyle\frac{1}{2}\sum_{\omega\in\mathbf{U}}\nu_{\omega}(M)\leq\sum_{\omega\in\mathbf{U}} P_\omega(M)=\sum_{\omega\in\mathbf{U}} Q_\omega(M)\leq n.\label{pro splitting number 13}\\
&\displaystyle i_{e^{\sqrt{-1}\hat{\theta}}}(\gamma)=i_1(\gamma)+\sum_{0\leq\theta<\hat{\theta}}S^+_M(e^{\sqrt{-1}\theta})
-\sum_{0<\theta\leq\hat{\theta}}S^-_M(e^{\sqrt{-1}\theta}).\label{pro splitting number 14}
\end{eqnarray}
\begin{equation}\label{pro splitting number 15}
\begin{array}{lll}
&\ \ \ \ \ \mu(Gr(e^{\sqrt{-1}\hat{\theta}}I_{2n}),Gr(P\gamma))-\mu(Gr(e^{\sqrt{-1}\bar{\theta}}I_{2n}),Gr(P\gamma))\\
=&\displaystyle \sum_{\bar{\theta}\leq\theta<\hat{\theta}}S^+_{PM}(e^{i\theta})-\sum_{\bar{\theta}<\theta\leq\hat{\theta}}S^-_{PM}(e^{i\theta})-
(\sum_{\bar{\theta}\leq\theta<\hat{\theta}}S^+_P(e^{i\theta})-\sum_{\bar{\theta}<\theta\leq\hat{\theta}}S^-_P(e^{i\theta})).
\end{array}
\end{equation}
\end{pro}
\begin{proof}
(\ref{pro splitting number 1},\ref{pro splitting number 2},\ref{pro splitting number 11}) follow from Lemma 9.1.6 and Lemma 9.1.9 in \cite{Y.Long.2002}. (\ref{pro Splitting number 8},\ref{pro splitting number 3}) follow from Lemma 1.8.14, Theorem 9.1.7 in \cite{Y.Long.2002} and the definition of Krein type numbers. (\ref{pro splitting number 5}-\ref{pro splitting number 7}) follow from List 9.1.12 in \cite{Y.Long.2002}. (\ref{pro similarity}) is a spacial case of Theorem 9.1.10(1). (\ref{pro 0<v-Spm(omega)<P,Q}) follows from Proposition 9.1.11 in \cite{Y.Long.2002}. (\ref{pro splitting number 12}) follows from Lemma 1.3.8 in \cite{Y.Long.2002}. (\ref{pro splitting number 13}) follows from the second identity of (\ref{pro splitting number 2}), (\ref{identity of Klein number}) and (\ref{pro splitting number 12}). (\ref{pro splitting number 14}) follows from Proposition 9.1.11 in \cite{Y.Long.2002}. (\ref{pro splitting number 15}) follows from (\ref{relation between two indies 2}) and (\ref{pro splitting number 14}).
\end{proof}
\begin{rem}\label{remark 2.1}
For any diagonalizable matrix $P\in \mathrm{Sp}(2n)$, we also have
$$S^+_{P}(\omega)=P_\omega(P),S^-_{P}(\omega)=Q_\omega(P),S^+_{P}(\omega)+S^-_{P}(\omega)=\nu_\omega(P), \forall \omega\in\mathbf{U}.$$
Indeed, let $\omega\in \mathbf{U}$. According to the diagnalizable condition, (\ref{identity of Klein number}), (\ref{pro splitting number 3}),(\ref{pro Splitting number 8}) and (\ref{pro splitting number 11}).
\begin{eqnarray*}
\nu_\omega(P)\leq P_\omega(P)+S^-_{P}(\omega)\leq P_\omega(P)+Q_\omega(P)=\dim\ker(P-\omega I_{2n})^{2n}=\nu_\omega(P).
\end{eqnarray*}
Then we have $S^-_{P}(\omega)=Q_\omega(P)$. Similarly, $S^+_{P}(\omega)=P_\omega(P)$ and the last identity follows directly.
\end{rem}

\section{Special properties of $P$-symmetric hypersurfaces}
In this section, we first provide some useful properties of closed characteristics. Then based on the method in \cite{Liu.Long.Zhu.2002}, we conclude that the multiplicity problem, i.e. estimating the total number of geometrically distinct closed characteristics on $\Sigma$, is equivalent to the estimation (\ref{equ the estimation}).

Let $j_\Sigma:\mathbb{R}^{2n}\rightarrow\mathbb{R}$ be a gauge function of $\Sigma$ defined by
\begin{eqnarray*}
j_\Sigma(0)=0,& j_\Sigma(x)=\inf\{\lambda>0,\frac{x}{\lambda}\in C\},\ x\neq 0.
\end{eqnarray*}
Fix a constant $\alpha\in(1,2)$ and define the Hamiltonian function $H_{\alpha}:\mathbb{R}^{2n}\rightarrow[0,\infty)$ by
\begin{eqnarray}
H_\alpha(x)=j_\Sigma(x)^\alpha,\ \forall x\in\mathbb{R}^{2n}.\label{hamiltonian with lapha}
\end{eqnarray}
Note that $H_\alpha\in C^1(\mathbb{R}^{2n},\mathbb{R})\cap C^2(\mathbb{R}^{2n}\setminus\{0\},\mathbb{R})$ is convex and $\Sigma=H^{-1}_\alpha(1)$. Since the gradient of $H_{\alpha}$ on $\Sigma$ is normal and nonzero, then the problem (\ref{initial problem}) is equivalent to the following fixed-energy problem
\begin{equation}\label{clo_cha_fix energy problem}
\left\{
\begin{aligned}{}
&H_\alpha(x)=1,\\
&\dot{x}=JH'_\alpha(x),\\
&x(0)=x(\tau).
\end{aligned}\right.
\end{equation}
It's well know that the solutions of (\ref{clo_cha_fix energy problem}) and (\ref{initial problem}) are in one to one correspondence with each other. It does not depend on the particular choice of $\alpha$.

For any $(\tau,x)\in\mathcal{J}(\Sigma)$, we denote by $\gamma_x\in\mathcal{P}_\tau(2n)$ the fundamental solution of the linear system
\begin{eqnarray}
\dot{y}(t)=JH_\alpha''(x(t))y(t),\label{linear hamiltonian system}
\end{eqnarray}
which is the linearization of system (\ref{clo_cha_fix energy problem}) respect to $x$. We also call $\gamma_x$ the associated symplectic path of $(\tau,x)$.
Consider $P\in\mathrm{Sp}(2n)$ and $\Sigma\in\mathcal{H}_P(2n)$, it implies that $H_\alpha(x)=H_\alpha(Px)$, $\forall x\in\mathbb{R}^{2n}$. Then we have
\begin{eqnarray}
H'_\alpha(x)=P^TH'_\alpha(Px),H''_\alpha(x)=P^TH''_\alpha(Px)P.\label{H',H''}
\end{eqnarray}

Denote by $(l,k)$ the greatest common divisor of $l,k\in N$. $l$ and $k$ are relatively prime integers if $(l,k)=1$.
Then we have following properties, which have been considered by X. Hu and S. Sun in \cite{Hu.Sun.2009} for orthogonal case.
\begin{pro}\label{properties of symmetric orbit}
Let $(\tau,x)\in\mathcal{J}(\Sigma)$ be a closed characteristic, and let $\gamma_x$ be the associated symplectic path of $(\tau,x)$, then we have\\
(1) $(\tau,Px)\in \mathcal{J}(\Sigma)$ is also a closed characteristic.\\
(2) $\gamma_{Px}(t)=P\gamma_x(t)P^{-1}$, where $\gamma_{Px}$ is the associated symplectic path of $(\tau,Px)$.\\
(3) If there exist $a,b\in[0,\tau)(a< b)$ and a smallest integer $k>0$ such that $x(a)=Px(b)$ and $x(t)=P^kx(t),\forall t\in[0,\tau)$, then $(\tau,x)$ is $P$-symmetric and
$$x(t)=Px(t+\frac{l\tau}{k}),$$
where $(l,k)=1$.\\
(4) If $(\tau,x)$ is $P$-symmetric as in (3), then $\gamma_x(t+\frac{l\tau}{k})=P^{-1}\gamma_x(t)P\gamma_x(\frac{l\tau}{k})$. Especially, $$\gamma_x(l\tau)=P^{-k}(P\gamma_x(\frac{l\tau}{k}))^k.$$
\end{pro}
\begin{proof}
(1) Since $(\tau,x)\in\mathcal{J}(\Sigma)$, then $(\tau,x)$ solves the system (\ref{clo_cha_fix energy problem}), i.e.
\begin{eqnarray}
\dot{x}=JH'_\alpha(x),\label{dot(x)=JH'(x)}
x(0)=x(\tau).
\end{eqnarray}
$\Sigma\in\mathcal{H}_P(2n)$ implies that $Px\in\Sigma$ and $Px(0)=Px(\tau)$. By (\ref{H',H''}), we obtain
$$P\dot{x}=PJH'_\alpha(x)=PJP^TH'_\alpha(Px)=JH'_\alpha(Px).$$
Then (1) follows.\\
(2) Combining with $(\ref{H',H''})$ and identity $PJP^T=J$, it follows that
\begin{eqnarray}
\dot{\gamma}_x(t)=JH''_\alpha(x(t))\gamma_x(t)\label{equ linearize system}
\end{eqnarray}
implies that
$$P\dot{\gamma}_x(t)=JH''_\alpha(Px(t))P\gamma_x(t).$$
Since the fundamental solution starts from $I_{2n}$, we get (2).\\
(3) By condition $x(a)=Px(b)$ and the uniqueness of solution of (\ref{dot(x)=JH'(x)}), we have
$$x(t)=Px(t+b-a)=P^2x(t+2(b-a))=\cdots=P^kx(t+k(b-a)).$$
Since $x(t)=P^kx(t)$ with the smallest integer $k>0$, there exist $l<k$ such that $k(b-a)=l\tau$, i.e. $x(t)=Px(t+\frac{l\tau}{k})$, and $l,k$ are relatively prime integers. Otherwise, let $l=rl_1,k=rk_1, r>1$, we obtain $k_1<k$ such that
$$x(t)=Px(t+\frac{l_1\tau}{k_1})=\cdots=P^{k_1}x(t),$$
which is a contradiction.\\
(4) Replacing $t$ of (\ref{equ linearize system}) by $t+\frac{l\tau}{k}$, by (\ref{H',H''}) and (3) above, we get that
\begin{eqnarray*}
P\dot{\gamma}_x(t+\frac{l\tau}{k})=&PJH''_\alpha(x(t+\frac{l\tau}{k}))\gamma_x(t+\frac{l\tau}{k})\\
=&PJH''_\alpha(P^{-1}x(t))\gamma_x(t+\frac{l\tau}{k})\\
=&PJP^TH''_\alpha(x(t))P\gamma_x(t+\frac{l\tau}{k})\\
=&JH''_\alpha(x(t))P\gamma_x(t+\frac{l\tau}{k}).
\end{eqnarray*}
Since the fundamental solution starts from $I_{2n}$ again, we have
$$P\gamma(t+\frac{l\tau}{k})\gamma(\frac{l\tau}{k})^{-1}P^{-1}=\gamma(t),$$
then (4) follows.
\end{proof}
Note that (3) of Proposition \ref{properties of symmetric orbit} actually implies that there are only two possibilities for each closed characteristic $(\tau,x)$, $x(\mathbb{R})=Px(\mathbb{R})$ or $x(\mathbb{R})\cap Px(\mathbb{R})=\emptyset$.

Now, we will introduce the Common Jump Index Theorem since this theorem is so important in this paper. First we define the mean index of $\gamma\in\mathcal{P}_{\tau}(2n)$ by
$$\hat{i}_1(\gamma)=\displaystyle\lim_{m\rightarrow\infty}\frac{i_1(\gamma^m)}{m}=\frac{1}{2\pi}\int_{\mathbf{U}}i_\omega(\gamma)d\omega,$$
which is always a finite real number.

Consider the linear hamiltonian system (\ref{linear hamiltonian system}) where $(\tau,x)\in\mathcal{J}(\Sigma)$. Since $H_\alpha$ is autonomous and also convex, i.e. $H_\alpha''(x)>0$ on $\mathbb{R}\setminus\{0\}$, then by Proposition 15.1.3 in \cite{Y.Long.2002}, the associated symplectic path $\gamma_x$ satisfies
\begin{eqnarray}
\nu_1(\gamma_x)\geq 1,\ i_1(\gamma_x)=\sum_{0<s<1}\nu_1(\gamma_x(s\tau))+n\geq n.\label{index biger than n}
\end{eqnarray}
Then we have the following proposition.
\begin{pro}\label{lem the relation between index intervals}
Let $\Sigma\in\mathcal{H}(2n),\ (\tau,x)\in\mathcal{J}(\Sigma)$ and $m\in\mathbb{N}$, denote by $M=\gamma(\tau),i^m_x=i_1(\gamma^m_x)$, and $\nu^m_x=\nu_1(\gamma^m_x)$. Then
\begin{eqnarray}
&i^{m+1}_x-i^m_x\geq 2,\ i^{m+1}_x+\nu^{m+1}_x-1\geq i^{m+1}_x>i^m_x+\nu^m_x-1,\label{index interval}\\
&\hat{i}_1(\gamma)>2,\label{mean index bigger than zero}\\
&i^2_x+2S^+_{M^2}(1)-\nu^2_x\geq n.\label{i+2S-v with 2nd iteration}
\end{eqnarray}
\end{pro}
\begin{proof}
(\ref{index interval}) is followed by Theorem 10.2.4 in \cite{Y.Long.2002} and (\ref{index biger than n}). (\ref{mean index bigger than zero}) is followed by Corollary 8.3.2(2) in \cite{Y.Long.2002}. (\ref{i+2S-v with 2nd iteration}) is followed by Lemma 15.6.3 in \cite{Y.Long.2002}.
\end{proof}

Apart from the Maslov-type index above, Ekeland also provided a very important result by using Fadell-Rabinowitz $S^1$ index theory. Let $E_\alpha:=\{u\in L^\beta([0,1],\mathbb{R}^{2n})|\int^1_0udt=0\}$, $\alpha^{-1}+\beta^{-1}=1$ and $\alpha\in(1,2)$. In Chapter V of \cite{Ekeland.1990}, the Clarke-Ekeland dual action functional on $E_\alpha$ is defined by
$$f_\alpha(u)=\int^1_0(\frac{1}{2}Ju\cdot\Pi u+H^*_\alpha(-Ju))dt,$$
where $H_\alpha$ is defined as (\ref{hamiltonian with lapha}) and the operator $\Pi:E^\beta\rightarrow E^\alpha$ is defined by $\frac{d}{dt}\Pi u=u$ and $\int^1_0\Pi udt=0$ which is compact.
Note that for any $(\tau,x)\in\mathcal{J}(\Sigma)$ and $m\in\mathbb{N}$, the action functional $f_\alpha$ has the corresponding critical point
\begin{eqnarray}
u^x_m(t)=(m\tau)^{\frac{\alpha-1}{\alpha-2}}\dot{x}(m\tau t), t\in[0,1], \label{critical point with 1 period}
\end{eqnarray}
which is also a 1-periodic orbit of Hamiltonian system of $H_\alpha$.

Since there is a nature $S^1$ action on $E_\alpha$ by shifting, then the Fadell-Rabinowitz $S^1$-cohomology index, i.e. ``$\mathrm{ind}$'', for $S^1$-invariant subset $[f_\alpha]_c:=\{u\in E_\alpha,f(u)< c\}\subset E_\alpha$ is defined by
$$\mathrm{ind}([f_\alpha]_c):=\sup\{k|c_1(E_\alpha)^{k-1}=c_1(E_\alpha)\cup\cdots\cup c_1(E_\alpha)\neq0\},$$
where $E_\alpha$ is a principle $S^1$ bundle, $c_1(E_\alpha)$ denote the first Chern class and $\cup$ is the cup product. For further details of Fadell-Rabinowitz index theory, we refer to their original paper \cite{Fradell.Rabinowitz.1978}. Define a sequence of critical values $c_k$ as
$$c_k=\inf\{c<0,\mathrm{ind}([f_\alpha]_c)\geq k\},k\in\mathbb{N},$$
which satisfy
$$-\infty<\displaystyle\min_{u\in E_\alpha}f_\alpha(u)=c_1\leq\cdots\leq c_k\leq\cdots<0,$$
and $c_k\rightarrow 0$ as $k\rightarrow+\infty$. By Theorem V.3.4 in Ekeland's book \cite{Ekeland.1990}, it holds that for any $k\in \mathbb{N}$, there exist $(\tau,x)\in\mathcal{J}(\Sigma)$ and $m\in\mathbb{N}$ such that the critical point $u^x_m$ of $f_\alpha$ defined by (\ref{critical point with 1 period}) satisfy
\begin{eqnarray}
f'_\alpha(u^x_m)=0,\ f_\alpha(u^x_m)=c_k,\ i^m_x\leq 2k-2+n\leq i^m_x+\nu^m_x-1.\label{Ekeland's inequality}
\end{eqnarray}
Under this result, Long and Zhu proved the following lemma.
\begin{lem}\label{lem The existance of injection map Psi}
Suppose ${}^\#\hat{\mathcal{J}}(\Sigma)<+\infty$. There exists an injection map $\Psi: \mathbb{N}\rightarrow\hat{\mathcal{J}}(\Sigma)\times\mathbb{N}(k\mapsto([(\tau,x)],m))$ such that (\ref{Ekeland's inequality}) happens for $(\tau,x)$.
\end{lem}
\begin{proof}
This Lemma is followed by Lemma 15.3.5(i) in \cite{Y.Long.2002}.
\end{proof}
Now we show the famous Common Index Jump Theorem proved by Long and Zhu in \cite{Long.zhu.2002} also in Chapter 11 of Long's book \cite{Y.Long.2002}.
\begin{lem}\label{common index jump theorem}
For $k=1,\cdots,q$, let $\gamma_k\in\mathcal{P}_{\tau_k}(2n)$ be a finite family of symplectic paths, denote by $M_k=\gamma(\tau_k),i^m_k=i_1(\gamma^m_k)$, and $\nu^m_k=\nu_1(\gamma^m_k)$. Suppose
\begin{eqnarray*}
\hat{i}(\gamma_k)>0,\ \forall k=1,\cdots,q.
\end{eqnarray*}
Then there exist infinitely many $(N,m_1,\cdots,m_q)\in\mathbb{N}^{q+1}$ with $N\geq n$ such that
\begin{align*}
\nu^{2m_k-1}_k=\nu^1_k,\ &\nu^{2m_k+1}_k=\nu^1_k,\\
i^{2m_k-1}_k+\nu^{2m_k-1}_k=&2N-(i^1_k+2S^+_{M_k}(1)-\nu^1_k),\\
i^{2m_k+1}_k=&2N+i^1_k,\\
i^{2m_k}_k\geq 2N-n,\ &i^{2m_k}_k+\nu^{2m_k}_k\leq 2N+n,
\end{align*}
for $k=1,\cdots,q$.
\end{lem}
By using this Common Index Jump Theorem, we have the following Proposition. Note that this proof is basically using the approach in [15] with a small modification. However, we would like to show the details of this proof because the following Proposition is the key in this paper.
\begin{pro}\label{pro lower bound of close cha}
Let $P\in \mathrm{Sp}(2n)$, $\Sigma\in \mathcal{H}_P(2n)$ and assume ${}^\#\hat{\mathcal{J}}(\Sigma)<+\infty$. If there exists an integer $n_1>0$ such that for any prime symmetric closed characteristic $[(\tau,x)]\in\hat{\mathcal{J}}(\Sigma)$,
\begin{eqnarray}
i(\gamma_x)+2S^+_{\gamma_x(\tau)}(1)-\nu(\gamma_x)\geq n_1,\label{equ the estimation}
\end{eqnarray}
then $${}^\#\mathcal{J}(\Sigma)\geq[\frac{n_1+n}{2}].$$
\end{pro}
\begin{proof}
According to the assumption that the set $\hat{\mathcal{J}}(\Sigma)$ is finite and the fact that $x$ and $Px$ are overlapped or completely separated, we can denote $\hat{\mathcal{J}}(\Sigma)$ by
\begin{eqnarray}
\{[(\tau_1,x_1)],\cdots,[(\tau_p,x_p)]\}\cup\displaystyle\bigcup^q_{i=1}\{[(\tau_{p+i},x_{p+i})],[(\tau_{p+i},Px_{p+i})],\cdots,[(\tau_{p+i},P^{k_i}x_{p+i})]\}.
\label{equ hat(J)=(...)}
\end{eqnarray}
Where $\{[(\tau_j,x_j)]\}^p_{j=1}$ are  geometrically distinct $P$-symmetric closed characteristics and
$$\{[(\tau_{p+i},x_{p+i})],[(\tau_{p+i},Px_{p+i})],\cdots,[(\tau_{p+i},P^{k_i}x_{p+i})]\}^q_{i=1}$$
are distinct sets of $P$-asymmetric ones. Let $K$ be the total number of $P$-asymmetric closed characteristics. Since $k_i\geq 2$ for any $ i=1,\cdots,q$, then we have ${}^\#\hat{\mathcal{J}}(\Sigma)=p+K<+\infty,\  K=k_1+\cdots+k_q\geq2q$.

Let $i^m_j=i_1(\gamma^m_{x_j}),\nu^m_j=\nu_1(\gamma^m_{x_j}), M_j=\gamma_{x_j}(\tau_j), j\in\{1,\cdots,p+q\}$. By (\ref{mean index bigger than zero}), every closed characteristic on $\Sigma$ corresponds to a sympletic path with positive mean index. Applying the Common Index Jump Theorem, i.e. Lemma \ref{common index jump theorem}, to the associated symplectic paths of $$\{(\tau_1,x_1),\cdots,(\tau_{p+q},x_{p+q}),(2\tau_{p+1},x^2_{p+1}),\cdots,(2\tau_{p+q},x^2_{p+q})\},$$ we obtain infinite many $(N,m_1,\cdots,m_{p+2q})\in\mathbb{N}^{p+2q+1}, N\geq n$ such that
\begin{eqnarray}
&i^{2m_j+1}_j=2N+i^1_j,\ i^{2m_j-1}_j+\nu^{2m_j-1}_j=2N-(i^1_j+2S^+_{M_j}(1)-\nu^1_j),\label{1 to p+q,equ 1}\\
&i^{2m_j}_j\geq 2N-n,i^{2m_j}_j+\nu^{2m_j}_j\leq 2N+n,\label{1 to p+q,unequ 2}
\end{eqnarray}
$\forall j\in\{1,\cdots,p+q\},$ and
\begin{eqnarray}
&i^{4m_{p+q+j_1}+2}_{p+j_1}=2N+i^2_{p+j_1},\label{p+q to p+2q,equ 1}\\
&i^{4m_{p+q+j_1}-2}_{p+j_1}+\nu^{4m_{p+q+j_1}-2}_{p+j_1}=2N-(i^2_{p+j_1}+2S^+_{M^2_{p+j_1}}(1)-\nu^2_{p+j_1}),\label{p+q to p+2q,equ 2}\\
&i^{4m_{p+q+j_1}}_{p+j_1}\geq 2N-n,\ i^{4m_{p+q+j_1}}_{p+j_1}+\nu^{4m_{p+q+j_1}}_{p+j_1}\leq 2N+n,\label{p+q to p+2q,unequ 1}
\end{eqnarray}
$\forall j_1\in \{1,\cdots,q\}.$

\textbf{Claim 1:} $m_{p+j_1}=2m_{p+q+j_1}$ for any $ j_1\in\{1,\cdots,q\}$.

In fact, by using (\ref{1 to p+q,unequ 2}), (\ref{i+2S-v with 2nd iteration}), (\ref{p+q to p+2q,equ 2}) and (\ref{index biger than n}), we have
\begin{align*}
i^{2m_{p+j_1}}_{p+j_1}\geq& 2N-n\geq 2N-(i^2_{p+j_1}+2S^+_{M_{p+j_1}}(1)-\nu^2_{p+j_1})\\
=&i^{4m_{p+q+j_1}-2}_{p+j_1}+\nu^{4m_{p+q+j_1}-2}_{p+j_1}> i^{4m_{p+q+j_1}-2}_{p+j_1}.
\end{align*}
Then, by (\ref{index biger than n}), (\ref{1 to p+q,unequ 2}) and (\ref{p+q to p+2q,equ 1}),
\begin{align*}
i^{2m_{p+j_1}}_{p+j_1}<i^{2m_{p+j_1}}_{p+j_1}+\nu^{2m_{p+j_1}}_{p+j_1}\leq& 2N+n\leq2N+i^2_{p+j_1}=i_{p+j_1}^{4m_{p+q+j_1}+2}.
\end{align*}
Finally, by (\ref{index interval}), we get
$$4m_{p+q+j_1}-2<2m_{p+j_1}<4m_{p+q+j_1}+2.\Rightarrow m_{p+j_1}=2m_{p+q+j_1}.$$
The claim follows.

According to Lemma \ref{lem The existance of injection map Psi}, we get an injection map $\Psi: \mathbb{N}\rightarrow\hat{\mathcal{J}}(\Sigma)\times\mathbb{N}$. Let
$$\Psi(N-s+1):=([(\tau_{j(s)},x_{j(s)})],m(s)),\ s\in\{1,\cdots,n\},$$
such that
\begin{equation}\label{<=2N-2s+n<=}
i^{m(s)}_{j(s)}\leq 2N-2s+n\leq i^{m(s)}_{j(s)}+\nu^{m(s)}_{j(s)}-1.
\end{equation}
where $j(s)\in\{1,\cdots,p+q\},m(s)\in\mathbb{N}$.
Then from (\ref{<=2N-2s+n<=}), (\ref{index biger than n}) and (\ref{1 to p+q,equ 1}), we deduce that
\begin{equation}\label{m(s)<2m(j(s))+1}
i^{m(s)}_{j(s)}\leq 2N-2s+n<2N+n\leq2N+i^1_{j(s)}=i^{2m_{j(s)}+1}_{j(s)}.
\end{equation}
Let
\begin{equation}\label{S1,S2}
S_1=\{k\in\{1,\cdots,[\frac{n_1+n}{2}]\},1\leq j(k)\leq p\},\\
S_2=\{k\in\{1,\cdots,n\},p+1\leq j(k)\leq p+q\}.
\end{equation}
\textbf{Claim 2:} ${}^\#S_1\leq p$.

In fact, let $k\in S_1$, then $1\leq j(k)\leq p$. By (\ref{<=2N-2s+n<=}), (\ref{1 to p+q,equ 1}) and the assumption (\ref{equ the estimation}), it follows that
\begin{align*}
i^{m(k)}_{j(k)}+\nu^{m(k)}_{j(k)}-1\geq&\ 2N-2k+n\geq 2N+n-2(\frac{n_1+n}{2})=2N-n_1\\
\geq&\ 2N-(i^1_{j(k)}+2S^+_{M_{j(k)}}(1)-\nu^1_{j(k)})=i^{2m_{j(k)}-1}_{j(k)}+\nu^{2m_{j(k)}-1}_{j(k)}.
\end{align*}
According to (\ref{m(s)<2m(j(s))+1}) and (\ref{index interval}), we conclude that
\begin{eqnarray*}
2m_{j(k)}-1<m(k)<2m_{j(k)}+1\Rightarrow m(k)=2m_{j(k)}.
\end{eqnarray*}
Then $\Psi(N-k+1)=([(\tau_{j(k)},x_{j(k)})],2m_{j(k)})$. Since $\Psi$ is injective, by (\ref{S1,S2}), we have ${}^\#S_1\leq p$.\\
\textbf{Claim 3:} ${}^\# S_2\leq 2q$.

In fact, let $k\in S_2$, then $p+1\leq j(k)\leq p+q$. From (\ref{<=2N-2s+n<=}), (\ref{i+2S-v with 2nd iteration}), (\ref{p+q to p+2q,equ 2}) and Claim 1, we obtain
\begin{align*}
i^{m(k)}_{j(k)}+\nu^{m(k)}_{j(k)}-1\geq&\ 2N-2s+n\geq 2N-n\geq 2n-(i^2_{j(k)}+2S^+_{M^2_{j(k)}}(1)-\nu^2_{j(k)})\\
=&\ j^{4m_{q+j(k)}-2}_{j(k)}+\nu^{4m_{q+j(k)}-2}_{j(k)}=j^{2m_{j(k)}-2}_{j(k)}+\nu^{2m_{j(k)}-2}_{j(k)}.
\end{align*}
By (\ref{m(s)<2m(j(s))+1}) and (\ref{index interval}) again, we have
$$2m_{j(k)}-2<m(k)<2m_{j(k)}+1\Rightarrow m(k)\in\{2m_{j(k)}-1,2m_{j(k)}\}.$$
Since $\Psi$ is injective again, this claim follows.

Finally, combining with Claim 2, Claim 3 and (\ref{S1,S2}), we have
\begin{eqnarray*}
{}^\#\hat{\mathcal{J}}(\Sigma)=p+K\geq p+2q\geq {}^\#S_1+{}^\#S_2\geq[\frac{n_1+n}{2}].
\end{eqnarray*}
\end{proof}
\section{Iteration theory of Maslov index and the proof of the main theorem}
In this section, the symplectic matrix $P$, which $\Sigma$ is symmetric with, will be selected as a special class of symplectic matrix satisfies $P^m=I$ for some $m$. Under this case, we will use Bott-type iteration formulas for ($P,m$)-iteration paths in \cite{Liu.Tang.2015} to estimate the number $n_1$ in Proposition \ref{pro lower bound of close cha}. We provide this estimation in Theorem \ref{thm i+2S-v} below. In the end of this section, we will prove the main results. Note that, this special ($P,m$)-iteration was found by Dong and Long in \cite{Dong.Long.2004}.

First we will show some notations. Let
\begin{eqnarray}
\omega_k:=e^{\sqrt{-1}\theta_k}=e^{\frac{2k\pi\sqrt{-1}}{m}},\ k=0,\cdots, m-1.\label{def of omega_k}
\end{eqnarray}
$\sigma(M)$ denotes the spectrum set of matrix $M$.
Then we define
$$\Omega_m(2n):=\{P\in\mathrm{Sp}(2n)|P^m=I_{2n}\ and\ \sigma(P)=\{\omega_k,\bar{\omega}_k\},k\in\{1,\cdots,[\frac{m}{2}]\}\}, $$
$$\Omega_{m,k}(2n):=\{P\in\Omega_m(2n)|\omega_k\in\sigma(P)\},\ k=1,\cdots,[\frac{m}{2}],$$
$$\tilde{\Omega}^a_m(2n):=\{P\in\Omega_m(2n)|S^+_P(\omega)=S^-_P(\omega)+a,\forall \omega\in \mathbf{U}^+\}$$ and $$\tilde{\Omega}^a_{m,k}(2n):=\{P\in\Omega_{m,k}(2n)|S^+_P(\omega_k)=S^-_P(\omega_k)+a\},\ k=1,\cdots,[\frac{m}{2}],$$
where $a=-1,0,1$ and $\mathbf{U}^+:=\{e^{\sqrt{-1}\theta},\theta\in(0,\pi]\}$. For convenience, we denote by $$\tilde{\Omega}_m(2n)=\tilde{\Omega}^0_m(2n),\ \tilde{\Omega}_{m,k}(2n)=\tilde{\Omega}^0_{m,k}(2n).$$
\begin{rem}\label{remark 0}
Note that $P\in\Omega_m(2n)$, $m\geq 2$ if and only if $P$ is similar to the following matrix
\begin{eqnarray}
R(-\theta)^{\diamond n-S^-_P(\omega)}\diamond R(\theta)^{\diamond S^-_P(\omega)},\label{diagonalized form}
\end{eqnarray}
where $R(\theta)=\begin{pmatrix}\cos\theta& -\sin\theta\\ \sin\theta& \cos\theta\end{pmatrix}$, $e^{im\theta}=1$ and $\frac{\theta}{2\pi}\notin \mathbb{Z}$.

Indeed, $\Leftarrow$ follows directly. We only consider $\Rightarrow$. Denote the minimal polynomial of $P$ by $$f(P)=(P-\omega_k)^p(P-\omega_{m-k})^p,$$
for some positive integers $k\in\{1,\cdots,[\frac{m}{2}]\},p>0$. Since $P^m=I_{2n}$, this polynomial $f(y)$ should exactly divide polynomial $y^m-1$, which means $p$ must be 1. Then we have $P$ is diagonalizable and then (\ref{diagonalized form}) is followed by (\ref{pro splitting number 5}-\ref{pro similarity}).
\end{rem}
According to Remark \ref{remark 0} above, we know that $a=0\ ($resp., $\pm 1)$ implies that $n$ is even(resp., odd).

For any two matrices $M_1=\begin{pmatrix}A_1&A_2\\ A_3&A_4\end{pmatrix}_{2i\times 2i}$ and $M_2=\begin{pmatrix}B_1&B_2\\ B_3&B_4\end{pmatrix}_{2j\times 2j}$, we define the diamond product of them by the $2(i+j)\times 2(i+j)$ matrix
\begin{eqnarray}
M_1\diamond M_2=\begin{pmatrix}A_1&0&A_2&0\\0&B_1&0&B_2\\A_3&0&A_4&0\\0&B_3&0&B_4\end{pmatrix},\label{diamond product}
\end{eqnarray}
and denote by $M^{\diamond k}$ the $k$-fold diamond product $M\diamond \cdots \diamond M$.

According to the properties of $P$-symmetric closed characteristics given by Proposition \ref{properties of symmetric orbit}(4), we define the following ($P,m$)-iteration for symplectic pathes as \cite{Dong.Long.2004} did. Note that this is the definition in \cite{Liu.Tang.2015} if we replace $P$ by $P^{-1}$.
\begin{definition}\label{def of interation path}
Let $\gamma\in\mathcal{P}_{\tau}(2n)$, $P\in\mathrm{Sp}(2n)$. Define the $(P,m)$-iteration of $\gamma$ by
$$\gamma^{m,P}:=P^{-(j-1)}\gamma(t-(j-1)\tau)(P\gamma(\tau))^{j-1},t\in[(j-1)\tau,j\tau],\ \forall j\in\{1,\cdots,m\}.$$
\end{definition}
Then following Bott-type iteration formulas hold.
\begin{pro}\label{bott interation}
For any $P\in\mathrm{Sp}(2n), \gamma\in\mathcal{P}_\tau(2n),\omega_0\in\mathbf{U}$, and $m\in\mathbb{N}$, we have
\begin{equation}\label{interation of index}
\mu(Gr(\omega_0I_{2n}),Gr(P^m\gamma^{m,P}))=\sum_{\omega^m=\omega_0}\mu(Gr(\omega I_{2n}),Gr(P\gamma)),
\end{equation}
\begin{equation}\label{interation of nullity}
\nu_{\omega_0}(\overline{P^m\gamma^{m,P}})=\sum_{\omega^m=\omega_0}\nu_\omega(\overline{P\gamma}),
\end{equation}
\begin{equation}\label{interation of splitting number}
S^\pm_{P^m\gamma^{m,P}(m\tau)}(\omega_0)=\sum_{\omega^m=\omega_0}S^\pm_{P\gamma(\tau)}(\omega),
\end{equation}
where $\overline{\gamma_1}:=\gamma_1\ast\xi$ for any symplectic path $\gamma_1\in C([0,\tau],\mathrm{Sp}(2n))$ not starting from $I_{2n}$, $\xi$ is arbitrary in $\{\xi\in\mathcal{P}_\tau(2n)|\xi(\tau)=\gamma_1(0)\}$ and $\ast$ is defined as (\ref{path connecting}).
\end{pro}
\begin{proof}
In view of Proposition \ref{properties of symmetric orbit}(4), we have $\gamma^{m,P}(m\tau)=P^{-m}(P\gamma(\tau))^m$.
According to the relation (\ref{relation between two indies 2}) and Definition 2.7 of Maslov $(P,\omega)$-index $i^{P}_\omega(\gamma)$ in \cite{Liu.Tang.2015}, we have
$$\mu(Gr(\omega I_{2n}),Gr(P\gamma))=i_\omega(\overline{P\gamma})-i_\omega(\xi)=i^{P^{-1}}_\omega(\gamma),\ \overline{P\gamma}=P\gamma\ast\xi.$$
The iteration path $\gamma^m$ defined in (5.5) of \cite{Liu.Tang.2015} is exactly the iteration path $\gamma^{m,P^{-1}}$ defined above. It follows from Theorem 1.1 in \cite{Liu.Tang.2015} that
\begin{eqnarray*}
\mu(Gr(\omega_0I_{2n}),Gr(P^{-m}\gamma^{m,P^{-1}}))&=&\mu(Gr(\omega_0I_{2n}),Gr(P^{-m}\gamma^m))=i^{P^m}_{\omega_0}(\gamma^m)=\sum_{\omega^m=\omega_0}i^P_\omega(\gamma)\\
&=&\sum_{\omega^m=\omega_0}\mu(\omega I_{2n},Gr(P^{-1}\gamma)),
\end{eqnarray*}
after replacing $P^{-1}$ by $P$, (\ref{interation of index}) holds. By the definition of $\nu_\omega(\gamma)$ and Theorem 9.2.1(2) in \cite{Y.Long.2002},
\begin{eqnarray*}
\nu_{\omega_0}(\overline{P^m\gamma^{m,P}})&=&\nu_{\omega_0}(P^m\gamma^{m,P}(m\tau))=\nu_{\omega_0}((P\gamma(\tau))^m)
=\sum_{\omega^m=\omega_0}\nu_{\omega}(P\gamma(\tau))\\
&=&\sum_{\omega^m=\omega_0}\nu_\omega(\overline{P\gamma}),
\end{eqnarray*}
then (\ref{interation of nullity}) follows. For the last one, by using Theorem 9.2.4 in \cite{Y.Long.2002}, we have
\begin{eqnarray*}
S^\pm_{P^m\gamma^{m,P}(\tau)}(\omega_0)=S^\pm_{(P\gamma(\tau))^m}(\omega_0)
=\sum_{\omega^m=\omega_0}S^\pm_{P\gamma(\tau)}(\omega_0),
\end{eqnarray*}
thus (\ref{interation of splitting number}) follows directly.
\end{proof}
Note that, if $\omega_0=1$, $P$ is orthogonal and satisfies $P^m=I_{2n}$, then (\ref{interation of index}) will reduce to the Theorem 1.1 in \cite{Hu.Sun.2009}. Next, in order to calculate the Maslov index of symplectic paths which does not start from $I_{2n}$, we need to prove the following Lemma. Before this, we introduce the definition of the crossing form in \cite{Rabbin.Salamon.1993}.

Let $V,\ \Lambda(t)\in\mathrm{Lag}(n),\ t\in[a,b]$ be a Lagrangian and a Lagrangian path, respectively. For each pair $(V,\Lambda)$, we define the crossing form as
\begin{eqnarray}
\Gamma(\Lambda(t),V,t):=Q(\Lambda(t),\dot\Lambda(t))|_{\Lambda(t)\cap V},\ \forall t\in[a,b],\label{crossing form}
\end{eqnarray}
where, for some Lagrangian $W$, $$(Q(\Lambda(t_0),\dot{\Lambda}(t_0))u,v):=\frac{d}{dt}|_{t_0}\omega_{st}(v,w(t)),\ \forall v\in \Lambda(t_0),\ w(t)\in W,\ v+w(t)\in\Lambda(t),$$ which does not depend on the choice of $W$. Like Maslov index, the crossing form is also invariant under the action of $\mathrm{Sp}(2n)$, i.e. $\Gamma(\Phi\Lambda,\Psi V,t)=\Gamma(\Lambda,V,t),\ \forall\Phi\in\mathrm{Sp}(2n)$. If the Lagrangian path $\Lambda$ has only regular crossings which means the crossing form $\Gamma(\Lambda(t),V,t)$ is nondegenerate when $\Lambda(t)\cap V\neq\{0\}$, then \cite{Long.zhu.2000} provides the following relation
\begin{eqnarray}
\mu(V,\Lambda)=m^+(\Gamma(\Lambda(a)),V,a)+\sum_{0<t<1}\mathrm{Sign}(\Gamma(\Lambda(t),V,t))-m^-(\Gamma(\Lambda(b),V,b))),\label{index formula}
\end{eqnarray}
where $m^\pm(M)$ denote the dimensions of positive and negative definite subspace of matrix $M$, respectively. $\mathrm{Sign}(M)=m^+(M)-m^-(M)$ denotes the signature of $M$. Note that any Lagrangian path can be regularized by a small perturbation.

Now we show the crossing form in a special case.
\begin{pro}\label{crossing form formula}
Let $\gamma\in C([0,\tau],\mathrm{Sp}(2n)).$ If $\gamma(t_0)\in\mathrm{Sp}^0_\omega(2n)$ and $-J\dot{\gamma}(t_0)\gamma(t_0)$ is positive definite for some $t_0\in[0,\tau]$, then $\Gamma(Gr(\gamma(t_0)),Gr(\omega I_{2n}),t_0)$ is positive definite and
$$\mathrm{Sign}\ \Gamma(Gr(\gamma(t_0)),Gr(\omega I_{2n}),t_0)=\nu_\omega(\gamma(t_0)).$$
\end{pro}
\begin{proof}
First we know $Gr(\omega I_{2n}),\ Gr(\gamma(t_0))$ are Lagrangian in $(\mathbb{R}^{4n},\omega_{st}\oplus(-\omega_{st}))$. Let $E_\omega(\gamma(t_0))\neq \{0\}$ be the eigenvector space of $\omega$. Since $\gamma(t_0)\in\mathrm{Sp}^0_\omega(2n)$, then $E_\omega(\gamma(t_0))\cong Gr(\omega I_{2n})\cap Gr(\gamma(t_0))\neq\{0\}$. Let $Z(M)=\begin{pmatrix}M\\I_{2n}\end{pmatrix}:\mathbb{R}^{2n}\rightarrow\mathbb{R}^{4n}$ be the Lagrangian frame of $M\in\mathrm{Sp}(2n)$ in $(\mathbb{R}^{4n},\omega_{st}\oplus(-\omega_{st}))$ whose image is $Gr(M)$. Then for any $u\in E_\omega(\gamma(t_0))$ we have
\begin{eqnarray*}
(\Gamma(Gr(\gamma(t_0)))u,Gr(\omega I_{2n}),t_0)
&=&\frac{d}{dt}|_{t_0}(\omega_{st}\oplus(-\omega_{st}))(Z(\omega I_{2n})u,Z(\gamma(t_0))u)\\
&=&\frac{d}{dt}|_{t_0}(\begin{pmatrix}J&0\\0&-J\end{pmatrix}\begin{pmatrix}\omega I_{2n}\\I_{2n}\end{pmatrix}u,\begin{pmatrix}\gamma(t)\\I_{2n}\end{pmatrix}u)\\
&=&\frac{d}{dt}|_{t_0}(J(\omega I_{2n})u,\gamma(t)u)\\
&=&(J\gamma(t_0)u,\dot{\gamma}(t_0)u)\\
&=&(-J\dot{\gamma}(t_0)\gamma(t_0)^{-1}u,u)>0.
\end{eqnarray*}
Then $\Gamma(Gr(\gamma(t_0)),Gr(\omega I_{2n}),t_0)$ is positive definite and the signature is exactly $\dim(E_\omega(\gamma(t_0)))$. Therefore this proposition follows.
\end{proof}
\begin{lem}\label{lem the calc formula of (P,omega)-index}
Let $P\in\mathrm{Sp}(2n)$ and let symmetric matrix path $B|_{[0,\tau]}$ be positive definite. $\gamma\in\mathcal{P}_\tau(2n)$ denotes the fundamental solution of $\dot{y}(t)=JB(t)y(t)$, then
$$\mu(Gr(\omega I_{2n}),Gr(P\gamma))=\nu_\omega(P)+\displaystyle\sum_{0<t<\tau}\nu_{\omega}(P\gamma(t)).$$
\end{lem}
\begin{proof}
First we get
$$B_p(t)=-JP\dot{\gamma}(t)(P\gamma(t))^{-1}=(P^{-1})^TB(t)P^{-1},\ t\in[0,\tau],$$ which is positive definite. By Proposition \ref{crossing form formula}, it further implies that the crossing form is always nondegenerate, i.e. $P\gamma$ is regular. Then, by Lemma \ref{crossing form formula} and (\ref{index formula}), we have
\small
\begin{eqnarray*}
&&\mu(Gr(\omega I_{2n}),Gr(P\gamma))\\
&=&m^+\Gamma(Gr(P),Gr(\omega I_{2n}),0)+\sum_{0<t<\tau}\mathrm{Sign}\ \Gamma(Gr(P\gamma(t)),Gr(\omega I_{2n}),t)-m^-\Gamma(Gr(P\gamma(\tau)),Gr(\omega I_{2n}), \tau)\\
&=&\nu_\omega(P)+\sum_{0<t<\tau}\nu_\omega(P\gamma(t)).
\end{eqnarray*}
\normalsize
This lemma follows.
\end{proof}
Then we have following estimation results.
\begin{thm}\label{thm i+2S-v}
Let $P\in\Omega_m(2n)$ with prime number $m\geq2$. $(\tau,x)\in\mathcal{J}(\Sigma)$ is a prime closed characteristic and $\gamma_x$ is the associated symplectic path of $(\tau,x)$. Assume that $$x(t)=Px(t+\frac{\tau}{m}),\forall t\in[0,\tau],$$ then
\begin{eqnarray}
\mu(Gr(I_{2n}),Gr(\gamma_x))+2S^+_{\gamma_x(\tau)}(1)-\nu(\gamma_x)\geq 2n-S^-_P(\omega),\ \omega\in \sigma(P)\cap \mathbf{U}^+.\label{i+2S-v inequality 3}
\end{eqnarray}
Especially, if $P\in\tilde{\Omega}_m(2n)$, then
\begin{eqnarray}
\mu(Gr(I_{2n}),Gr(\gamma_x))+2S^+_{\gamma_x(\tau)}(1)-\nu(\gamma_x)\geq \frac{3n}{2},\label{i+2S-v inequality 1}
\end{eqnarray}
else if $P\in\tilde{\Omega}^1_m(2n)(resp.,\ \tilde{\Omega}^{-1}_m(2n))$, then
\begin{eqnarray}
\mu(Gr(I_{2n}),Gr(\gamma_x))+2S^+_{\gamma_x(\tau)}(1)-\nu(\gamma_x)\geq [\frac{3n}{2}]+1(resp.,\ [\frac{3n}{2}]).\label{i+2S-v inequality 4}
\end{eqnarray}
\end{thm}
\begin{proof}
Denote by $\omega:=e^{\sqrt{-1}\theta}$ and $\omega_k$ as (\ref{def of omega_k}). Without loss of generality, we assume $P\in\Omega_{m,k}(2n)$, i.e. $\{\omega_k,\omega_{m-k}\}=\sigma(P)$ where $k\in\{1,\cdots,[\frac{m}{2}]\}$. Let $M=\gamma(\frac{\tau}{m})$, $\hat{\gamma}_x(t)=\gamma_{x}(t),\ \forall t\in[0,\frac{\tau}{m}]$. According to Proposition \ref{properties of symmetric orbit} and Definition \ref{def of interation path}, we have $\gamma_x=\hat{\gamma}^{m,P}_x$ with end point $\gamma_x(\tau)=(PM)^m$.

(1) When $m>2$, we know that $k-1,k,m-k$ are all different. In view of Proposition \ref{bott interation},\ Remark \ref{remark 2.1}, (\ref{pro splitting number 1},\ref{pro Splitting number 8},\ref{pro 0<v-Spm(omega)<P,Q},\ref{pro splitting number 13},\ref{pro splitting number 15}), Lemma \ref{lem the calc formula of (P,omega)-index} and the convexity of $\Sigma$, we have following facts
\begin{eqnarray*}
&(\ref{pro splitting number 1}),(\ref{pro Splitting number 8}),(\ref{pro 0<v-Spm(omega)<P,Q})\Rightarrow 0\leq S^\pm_N(\omega)\leq\nu_\omega(N)\leq P_\omega(N)+S^-_N(\omega)\ \&\ Q_\omega(N)+S^+_N(\omega),\\
&S^+_N(\omega)\leq P_\omega(N),\ S^-_N(\omega)\leq Q_\omega(N),\forall N\in\mathrm{Sp}(2n).\\
&P\in\Omega_{m,k}(2n),Remark\ \ref{remark 2.1},Remark\ \ref{remark 0}\Rightarrow P\ is\ diagonalizable\ and\  S^+_P(\omega_k)=P_{\omega_k}(P),\\
&S^-_P(\omega_k)=Q_{\omega_k}(P),S^+_P(\omega_k)+S^-_P(\omega_k)=\nu_{\omega_k}(P)=n.\\
&P\in\Omega_{m,k}(2n),H_\alpha''(x)>0,Lemma\ \ref{lem the calc formula of (P,omega)-index}\Rightarrow
\mu(Gr(\omega_k I_{2n}),Gr(P\gamma))\geq \nu_{\omega_k}(P)=n.\\
&P\in\Omega_{m,k}(2n),(\ref{pro splitting number 15})\Rightarrow \mu(Gr(\omega_{k} I_{2n}),Gr(P\hat{\gamma}_x))-\mu(Gr(\omega_{k-1} I_{2n}),Gr(P\hat{\gamma}_x))\\
&=\displaystyle\sum_{\theta_{k-1}\leq\theta<\theta_k}S^+_{PM}(\omega)-\sum_{\theta_{k-1}<\theta\leq\theta_k}S^-_{PM}(\omega)
-S^+_P(\omega_{k-1})+S^-_P(\omega_k).\\
&Proposition\ \ref{bott interation},(\ref{pro 0<v-Spm(omega)<P,Q}),(\ref{pro splitting number 13})\Rightarrow \nu_1(\gamma_x)-S^+_{(PM)^m}(1)\leq\displaystyle\sum^{m-1}_{i=0}Q_{\omega_i}(PM)=\sum^{m-1}_{i=0}P_{\omega_i}(PM)\leq n.
\end{eqnarray*}
With these facts, we can deduce that
\begin{eqnarray*}
&\mu(Gr(I_{2n}),Gr(\gamma_x))+2S^+_{\gamma_x(\tau)}(1)-\nu_1(\gamma_x)\\
=&\displaystyle\sum^{m-1}_{i=0}\mu(Gr(\omega_i I_{2n}),Gr(P\hat{\gamma}_x))+2S^+_{(PM)^m}(1)-\nu_1(\gamma_x)\\
\geq&\mu(Gr(\omega_{k-1} I_{2n}),Gr(P\hat{\gamma}_x))+\mu(Gr(\omega_k I_{2n}),Gr(P\hat{\gamma}_x))+\mu(Gr(\omega_{m-k} I_{2n}),Gr(P\hat{\gamma}_x))\\
&\displaystyle+2S^+_{(PM)^m}(1)-\nu_1(\gamma_x)\\
\geq& 2\nu_{\omega_k}(P)+\nu_{\omega_{m-k}}(P)
+\mu(Gr(\omega_{k-1} I_{2n}),Gr(P\hat{\gamma}_x))-\mu(Gr(\omega_k I_{2n}),Gr(P\hat{\gamma}_x))\\
&+2S^+_{(PM)^m}(1)-\nu_1(\gamma_x)\\
=&\displaystyle3n-(\sum_{\theta_{k-1}\leq\theta<\theta_k}S^+_{PM}(\omega)-\sum_{\theta_{k-1}<\theta\leq\theta_k}S^-_{PM}(\omega))+S^+_P(\omega_{k-1})-S^-_P(\omega_k)\\
&\displaystyle+\sum^{m-1}_{i=0}S^+_{PM}(\omega_i)+S^+_{(PM)^m}(1)-\nu_1(\gamma_x)\\
\geq& \displaystyle 3n-S^-_P(\omega_k)-\sum_{\theta_{k-1}<\theta<\theta_k}S^+_{PM}(\omega)
-\sum^{m-1}_{i=0}P_{\omega_i}(PM)\\
\geq& \displaystyle 3n-S^-_P(\omega_k)-\sum_{\omega\in U}P_{\omega}(PM)\ \geq\ 2n-S^-_P(\omega_k).
\end{eqnarray*}
When $m=2$, $P$ must be $-I_{2n}$. Similarly, we can deduce that
\begin{eqnarray}
&\mu(Gr(I_{2n}),Gr(\gamma_x))+2S^+_{\gamma_x(\tau)}(1)-\nu(\gamma_x)\geq 2n.\label{i+2S-v inequality m=2}
\end{eqnarray}
By Remark \ref{remark 2.1}, (\ref{pro splitting number 2}), $P\in\tilde{\Omega}_{m,k}(2n)$ implies that $n$ is even and $$S^\pm_P(\omega_k)=S^\mp_P(\omega_{m-k})=\frac{n}{2}.$$
Thus we get (\ref{i+2S-v inequality 1}). Similarly, $P\in\tilde{\Omega}^1_{m,k}(2n)(resp.,\ P\in\tilde{\Omega}^{-1}_{m,k}(2n))$ implies that $n$ is odd and
$$S^-_P(\omega_k)=[\frac{n}{2}](resp.,\ [\frac{n}{2}]+1).$$
Then (\ref{i+2S-v inequality 4}) follows.
\end{proof}
Here we note that, when $m=2$, the calculation in the above proof essentially coincides with the proof of C. Liu, Y. Long and C. Zhu in \cite{Liu.Long.Zhu.2002}. Therefore, we omit the calculation.

Then we conclude the following Theorem.
\begin{thm}\label{main theorem 1}
Let $P\in\Omega_m(2n)$ with an integer $m\geq 2$ and $\Sigma\in\mathcal{H}_P(2n)$. It holds that
\begin{eqnarray}\label{equ1 thm1.1}
{}^\#\hat{\mathcal{J}}(\Sigma)\geq [\frac{n_1+n}{2}]=n+[\frac{-\max\{S^-_{P}(\omega),S^-_P(\bar{\omega})\}}{2}].
\end{eqnarray}
Especially, if $m$ is even, then
\begin{eqnarray}\label{equ3 thm1.1}
{}^\#\hat{\mathcal{J}}(\Sigma)\geq n.
\end{eqnarray}
\end{thm}
\begin{proof}
We prove this result by following steps.\\
\textbf{Step 1}: Assume that $m$ is not prime. Denote by $m=m_1p$ and $P_1=P^{m_1}$, where $p$ is a prime factor. Then we have $\Sigma=P_1^p\Sigma$, and it is sufficient to consider $m$ as a prime number. \\
\textbf{Step 2}: Let $m$ be a prime. Assume that $P\in\Omega_{m,k}(2n)$ and $(\tau,x)\in\mathcal{J}_P(\Sigma)$ is a prime closed
characteristic. By Proposition \ref{properties of symmetric orbit}(3), there exist an integer $l\in\{1,\cdots,m-1\}$ such that
 $$x(t)=Px(t+\frac{l\tau}{m}),\ \forall t\in \mathbb{R}.$$ Then we can choose $r\in\{1,\cdots,m-1\}$ such that $$x(t)=P^rx(t+\frac{\tau}{m}),\ \forall t\in \mathbb{R}.$$
By Remark \ref{remark 0}, $P^r$ will be similar to $R(-r\theta)^{\diamond S^+_P(\omega)}\diamond R(r\theta)^{\diamond S^-_P(\omega)},$ where $\omega=e^{\sqrt{-1}\theta}$. It implies
\begin{eqnarray}
S^\pm_P(\omega)=S^\pm_{P^r}(\omega^r).\label{P and Pr}
\end{eqnarray}
If ${}^\#\hat{J}(\Sigma)=+\infty$, we are done. Then we assume that ${}^\#\hat{J}(\Sigma)<+\infty$ and $m>2$. If $\omega_k^r\in\{\omega_1,\cdots,\omega_{[\frac{m}{2}]}\}$, then by (\ref{i+2S-v inequality 3}), (\ref{relation between two indies}) and (\ref{P and Pr}), we obtain that
\begin{eqnarray}\label{inequ 1}
i_1(\gamma_x)+2S^+_{\gamma_x(\tau)}(1)-\nu_1(\gamma_x)\geq n-S^-_{P^r}(\omega_k^r)\geq n-\max\{S^-_{P}(\omega_k),S^-_P(\omega_{m-k})\},
\end{eqnarray}
else if $\omega_{m-k}^r\in\{\omega_1,\cdots,\omega_{[\frac{m}{2}]}\}$, (\ref{inequ 1}) follows as well. Since (\ref{inequ 1}) does not depend on the choice of $(\tau,x)$, we have
$$n_1=n-\max\{S^-_{P}(\omega_k),S^-_P(\omega_{m-k})\}.$$
Then by Proposition \ref{pro lower bound of close cha}, (\ref{equ1 thm1.1}) follows. However, if $m=2$, it follows from (\ref{i+2S-v inequality m=2}) that $n_1=n$. Then combining with Step 1 and Proposition \ref{pro lower bound of close cha}, this theorem holds.
\end{proof}

Then we prove Theorem \ref{main theorem} as follows.
\begin{proof}
When $m$ is even, by Remark \ref{remark 0} and (\ref{equ3 thm1.1}), (\ref{equ2 thm1.1}) follows. We only focus on $m>2$ is odd. If $n$ is even, it follows from the assumption (\ref{assumption 1}) and Remark \ref{remark 0} that $P\in\tilde{\Omega}_{m,k}(2n)$ for some $k\in\{1,\cdots, [\frac{m}{2}]\}$. By Remark \ref{remark 0} and Theorem \ref{main theorem 1},  we have $S^\pm_P(\omega_k)=S^\pm_P(\omega_{m-k})=n_1=\frac{n}{2}$ and (\ref{equ2 thm1.1}) holds. Similarly, If $n$ is odd, $P\in\tilde{\Omega}^{\pm1}_{m,k}(2n)$ for some $k\in\{1,\cdots, [\frac{m}{2}]\}$. By Remark \ref{remark 0} and Theorem \ref{main theorem 1} again, we get $\max(S^+_P(\omega_k),S^-_P(\omega_k))=[\frac{n}{2}]+1$ and $n_1=[\frac{n}{2}]$, then ${}^\#\hat{\mathcal{J}}(\Sigma)\geq [\frac{3n-1}{4}]$. Since $n$ is odd, it holds that $[\frac{3n}{4}]=[\frac{3n-1}{4}]$ and this result follows.
\end{proof}
Proof of Theorem \ref{main theorem 2}.
\begin{proof}
If $m=2$, $P$ is similar to $-I_{2n}$, then this theorem follows from Theorem \ref{main theorem}. Let $m>2$ and $e^{i\theta}=\omega$. From assumption (i) and Remark \ref{remark 0}, we know that $S^-_P(\omega)=0$ and $S^+_P(\omega)=n$. According to (ii), (\ref{relation between two indies}),  and (\ref{i+2S-v inequality 3}), we have $n_1=n$. Then by Proposition \ref{pro lower bound of close cha}, this result follows.
\end{proof}

\label{lastpage}\vspace{3mm}
\noindent
\textbf{Acknowledgement}

\vspace{3mm}
This research would not be possible without the support of many people. The authors are grateful to Professor Yiming Long for his interest and Professor Xijun Hu for many useful advises and patient guidance. Finally, the authors won't forget their beloved friends and family members, for their understanding and endless love through the duration of their studies.
\vspace{3mm}

\noindent


\end{document}